\documentclass[final,floating,letterpaper]{siamltex1213}
\usepackage{}
\usepackage{latexsym,amsmath,amssymb}
\allowdisplaybreaks[4]

\usepackage{comment}
\usepackage[compress]{cite}
\usepackage{algorithm}
\usepackage{algorithmic}

\algsetup{indent=2em}

\usepackage{bbm}
\usepackage{bm}
\usepackage{mathbbold}
\usepackage{amsfonts}
\usepackage[T1]{fontenc}
\usepackage{latexsym,amssymb,amsmath,amsfonts}
\usepackage{graphicx}
\usepackage{subfigure}
\usepackage{epsf}
\usepackage{epstopdf}
\usepackage{capt-of}
\usepackage{hyperref}

\usepackage{todonotes}
\usepackage[normalem]{ulem}

\renewcommand{\theequation}{\arabic{section}.\arabic{equation}}

\newcommand{\R}{{\mathbb R}}
\newcommand{\Z}{{\mathbb Z}}

\newcommand{\C}{{\mathbb C}}

\newcommand{\be}{\begin{eqnarray}}
\newcommand{\ben}{\begin{eqnarray*}}
\newcommand{\en}{\end{eqnarray}}
\newcommand{\enn}{\end{eqnarray*}}
\newcommand{\ba}{\backslash}
\newcommand{\pa}{\partial}

\newcommand{\ov}{\overline}
\newcommand{\curl}{{\rm curl\,}}

\newcommand{\G}{\Gamma}
\newcommand{\eps}{\epsilon}

\newcommand{\om}{\omega}

\newcommand{\hx}{\hat{x}}
\newcommand{\bbS}{\mathbb{S}^2}
\newcommand{\bE}{\mathbf{E}}
\newcommand{\bH}{\mathbf{H}}
\newcommand{\bd}{\mathbf{d}}
\newcommand{\bp}{\mathbf{p}}
\newcommand{\bx}{\mathbf{x}}
\newcommand{\by}{\mathbf{y}}
\newcommand{\bz}{\mathbf{z}}

\newcommand{\I}{\ensuremath{\mathrm{i}}}

\newcommand{\D}{\ensuremath{\mathrm{d}}}

\DeclareMathOperator{\EP}{EP}
\DeclareMathOperator{\MP}{MP}

\DeclareMathOperator{\Dcurl}{curl}
\DeclareMathOperator{\Dgrad}{grad}
\DeclareMathOperator{\DDiv}{Div}

\title{Inverse electromagnetic obstacle scattering problems with multi-frequency sparse backscattering far field data}

\author{Tilo Arens\thanks{Department of Mathematics, Karlsruhe Institute of Technology, 76131 Karlsruhe, Germany ({\tt tilo.arens@kit.edu}).}
\and
Xia Ji\thanks{LSEC, Academy of Mathematics and Systems Science, Chinese Academy of Sciences, 100190 Beijing, China ({\tt jixia@lsec.cc.ac.cn}).}
\and
Xiaodong Liu\thanks{Institute of Applied Mathematics, Academy of Mathematics and Systems Science, Chinese Academy of Sciences, 100190 Beijing, China({\tt xdliu@amt.ac.cn}).}
}
\begin{document}

\maketitle

\begin{abstract}
This paper is dedicated to design a direct sampling method of inverse electromagnetic scattering problems, which uses multi-frequency sparse backscattering far field data for reconstructing the boundary of perfectly conducting obstacles. We show that a smallest strip containing the unknown object can be approximately determined by the multi-frequency backscattering far field data at two opposite observation directions. The proof is based on the Kirchhoff approximation and Fourier transform. Such a strip is then reconstructed by an indicator, which is the absolute value of an integral of the product of the data and some properly chosen function over the frequency interval. With the increase of the number of the backscattering data, the location and shape of the underlying object can be reconstructed.
Numerical examples are conducted to show the validity and robustness of the proposed sampling method. The numerical examples also show that the concave part of the underlying object can be well reconstructed, and the different connected components of the underlying object can be well separated.

\vspace{.2in} {\bf Keywords:}
Electromagnetic obstacle scattering, sparse backscattering data, uniqueness, direct sampling methods.

\vspace{.2in} {\bf AMS subject classifications:}
35R30, 35P25, 78A46

\end{abstract}
\section{Introduction}

Scattering of electromagnetic waves plays an important role in many fields of applied mathematics such as radar, medical imaging, nondestructive testing and geophysical exploration. The inverse scattering problem considered in this paper is concerned with the reconstruction of the boundary of an unknown object from the knowledge of the incident time-harmonic plane wave and of the far field pattern associated with the scattered field. Many research papers have been dedicated to the formulation and analysis of the over-determined problem at a fixed frequency where measurements are taken for all directions of incidence of plane waves and for all observation directions. Mathematically, much is known for such full aperture data: Uniqueness of solution for the inverse problem can then be established and many reconstruction algorithms have been proposed \cite{CK}. We refer to the monograph \cite{CakoniColtonMonk} for the well developed linear sampling method in inverse electromagnetic scattering problems using full aperture data and to Chapter 5 of \cite{KirschGrinberg} for the factorization method for inverse scattering of electromagnetic waves by a penetrable object consisting of an inhomogeneous medium.

However, in many applications it is difficult to conduct an experiment in which measurements are taken simultaneously in all observation directions around an unknown scatterer. Thus, limited aperture problems arise in many practical applications. Indeed, for the linear sampling method limited aperture data can present a severe challenge, see Section 3.4 of \cite{CakoniColtonMonk}. Instead of developing methods using limited-aperture data, an alternative approach is to recover the data that cannot be measured directly. Subsequently, methods using full aperture data can be employed. We refer to \cite{JiLiuXi, LiuSun} for some data retrieval techniques along these lines.

From the practical point of view, backscattering far field data, where the observation direction and the incident direction are opposite to each other, is of considerable interest. To make the inverse scattering problem solvable, multiple frequencies have to be used. The inverse backscattering problem has attracted and challenged mathematicians in the last few years \cite{EskinRalston,HaddarKusiakSylvester, HahnerKress,LiLiuWang,NSHS, Shin} .

In this work, we consider a backscattering obstacle problem with the multi-frequency far field data at sparsely distributed observation directions. In particular, we are interested in what kind of information about the unknown object can be determined by the multi-frequency backscattering far field data at two observation directions. This work is a generalization of the recent work on inverse electromagnetic source scattering problem \cite{JL-Maxwell-source}, where the multi-frequency far field data at sparsely distributed directions are used.
However, the inverse source problem is linear while the inverse obstacle problem is nonlinear, making the generalization nontrivial and difficult. Our approach is to consider the physical optics approximation at high frequency. We show that a particular strip containing the unknown objects can be approximately determined from the multi-frequency backscattering far field data at two different observation directions. Furthermore, a rough convex support of the underlying object can be well captured by the multi-frequency backscattering far field data at three pairs of opposite observation directions.

The second contribution of this paper is to introduce a direct sampling method for the boundary reconstruction. The indictor is defined by an integral of the product of the backscattering far field data and an exponential function over a finite frequency band. This idea originates in the direct sampling method for inverse source scattering problems \cite{AHLS, JL-elastic, JL-Maxwell-source}. We also refer to \cite{JL-Acoustic-sparse} for a corresponding study for the acoustic obstacle reconstruction problem.
Numerical examples show that the strip containing the object can be reconstructed by the multi-frequency backscattering far field data at two observation directions.
With the increase of the observation directions, the shape can further be reconstructed, even the concave part.

This paper is organized as follows. In the next section, we fix the notations and formulate the direct electromagnetic obstacle scattering problem. We also precisely formulate the inverse problem to be solved. In section \ref{Sec-POA}, we recall the physical optics approximation and use it to derive an approximation to the far field pattern valid at high frequencies. We proceed to establish a uniqueness result for the inverse problem in section \ref{Sec-Uni}. Section \ref{Sec-DSM} is devoted to propose a direct sampling method and give an explanation why this method works. Finally, in section \ref{Sec-Num}, some numerical simulations are presented to validate the effectiveness and robustness of the proposed direct sampling method.

\section{Inverse electromagnetic scattering from a perfect conductor}\label{ESPC}

In this section, we set the stage by presenting the time harmonic electromagnetic obstacle scattering problem under consideration. Let us begin with the notations used throughout this paper. Vectors are distinguished from scalars by the use of bold typeface.
For a vector $\mathbf{a}:=(a_1, a_2, a_3)^{\rm T}\in\C^3$, where the superscript $``{\rm T}"$ denotes the transpose, we define the Euclidean norm of $\mathbf{a}$ by $|\mathbf{a}|:=\sqrt{\mathbf{a\cdot \ov{a}}}$, where
$\ov{\mathbf{a}}:=(\ov{a_1}, \ov{a_2}, \ov{a_3})^{\rm T}\in\C^3$ and $\ov{a_j}$ is the complex conjugate of $a_j$.
Denote by $\bbS:=\{\mathbf{x}\in \R^3: |\mathbf{x}|=1\}$ the unit sphere in $\R^3$.

We consider electromagnetic wave propagation in a non-conductive isotropic homogeneous medium in $\R^3$ with electric permittivity $\eps$ and magnetic permeability $\mu$. Define $k:=\om\sqrt{\eps\mu}$ to be the wave number at frequency $\om>0$. The incident field of particular interest is the electromagnetic
plane wave
\be\label{EiHi}
\bE^{i}({\bf x,d,p},k)= \bp \, e^{ik\bx\cdot \bd},\quad \bH^{i}({\bf x,d,p},k)=(\bd\times \bp) \, e^{ik\bx\cdot \bd},\quad \bx\in\R^3, \bd,\bp\in \bbS,\qquad
\en
\noindent where the unit vector $\bd$ describes the direction of propagation, and the polarization vector $\bp$ must be orthogonal to the direction of propagation, so $\bd\cdot\bp=0$.
Let $D$ be a bounded domain in $\R^3$ with Lipschitz boundary $\pa D$ such that the exterior
$D^{e}:=\R^3\ba\ov{D}$ of $D$ is connected.
The scatterer $D$ gives rise to a pair of scattered electromagnetic fields
$(\bE^s, \bH^s)\in H_{\text{loc}}(\curl, D^{e})\times H_{\text{loc}}(\curl, D^{e})$ which satisfy the
time-harmonic Maxwell equations \cite{KirschHettlich}
\be\label{EH_Maxwellequations}
\curl \bE^s-ik\bH^s=0,\quad \curl \bH^s+ik\bE^s=0\qquad\text{in } D^{e}.
\en
For a perfect conductor, we impose the perfect conducting boundary condition, i.e.,
\be\label{pec}
\bm{\bm{\nu}} \times \bE=0 \qquad \text{on } \pa D,
\en
where $\bm{\nu}$ is the unit outward normal to $\pa D$ and $\bE:=\bE^i+\bE^s$ is the total electric field.
The scattered fields $(\bE^s, \bH^s)$ are out-going. Mathematically, this means that the scattered field
$(\bE^s, \bH^s)$ satisfies the Silver-M\"{u}ller radiation condition
\be\label{SMrc}
\lim_{| \bx | \to \infty}(\bH^s \times \bx - | \bx | \,  \bE^s ) = 0
\en
where the limit is attained uniformly in all directions ${\bf\hx}:=\bx/ | \bx |$.
It is well known that there exists a unique solution $(\bE^s, \bH^s)\in H_{\text{loc}}(\curl, D^{e})\times H_{\text{loc}}(\curl, D^{e})$
of the scattering system (\ref{EH_Maxwellequations})-(\ref{SMrc}) (see e.g., \cite{Monk}).

It is also well known (see e.g., \cite{CK,Monk}) that every radiating solution of \eqref{EH_Maxwellequations} has an asymptotic behavior of the form
\be\label{Easy}
{\bf E^s}(\bx, \bd, \bp,k)=\frac{e^{ik|{\bf x}|}}{4\pi|{\bf x}|}{\bf E}^{\infty}({\bf\hx}, \bd, \bp, k) + O\left(\frac{1}{|{\bf x}|^2}\right),\quad |{\bf x}|\rightarrow\infty,
\en
uniformly w.r.t. ${\bf\hx}$. The vector field ${\bf E}^{\infty}({\bf\hx}, \bd, \bp, k)$ is known as the electric far field pattern of the electric scattered field $\bE^s(\bx, \bd, \bp,k)$. It is an analytic function on the unit sphere $\bbS$ with respect to observation direction ${\bf \hx}$ and it is a tangential field, i.e., ${\bf \hx\cdot E^{\infty}}({\bf\hx}, \bd, \bp, k)=0$ for all ${\bf\hx}\in\bbS$.
For the scattering of plane waves by a perfect conductor $D$, the far field pattern is given by
\be
\label{HuygenPrinciple} 
\bE^{\infty}({\bf\hx}, \bd, \bp, k) 
&=&ik{\bf\hx}\times\int_{\pa D} [\bm{\nu}(\by)\times \bH(\by)]\times {\bf\hx} \, e^{-ik{\bf\hx}\cdot\by} \, ds(\by), \quad {\bf\hx}, \bd, \bp\in\bbS .
\en
This equation follows from the Stratton-Chu representation formula \cite{CK}.

Fixing two wave numbers $0<k_{min}<k_{max}$, we consider the the scattering system (\ref{EH_Maxwellequations})-(\ref{SMrc}) with
\be\label{kassumption}
k\in K:= [k_{min}, k_{max}] \, .
\en
This paper is concerned with the following inverse problem:

\smallskip


\begin{quote}
  {\bf IP:} {\em Find the location and shape of the obstacle $D$ from the knowledge of the electric far field pattern ${\bf E}^{\infty}({\bf\hx}, \bd, \bp, k)$, $k\in K$, at finitely many observation directions ${\bf\hx}$ for finitely many incident directions $\bd$ with corresponding polarizations $\bp$}.
\end{quote}

In particular, ${\bf E}^{\infty}(-\bd, \bd, \bp, k)$ denotes the backscattering electric far field pattern.
For some positive integer $l\in\Z$, define
\ben
\Theta_l:=\{ \pm\bbtheta_1,  \pm\bbtheta_2, \cdots, \pm\bbtheta_l | \, \bbtheta_j\in \bbS, j=1,2,\cdots,l\},
\enn
which is a subset of $\bbS$ with finitely many directions. To each $\bd \in \Theta_l$, a polarization $\bp(\bd)$ is assigned. {\bf\em The inverse backscattering problem consists in the determination of $D$
from ${\bf E}^{\infty}(-\bd, \bd, \bp(\bd), k)$ for all $\bd\in \Theta_l$ and all $k$ in a bounded band $K$.}
That is, roughly speaking, whether we can determine the location and shape of the obstacle $D$ by measuring the echoes produced by
incident plane waves in the directions $\bd\in\Theta_l$.

\section{Physical optics approximation}\label{Sec-POA}

Consider the scattering of a plane wave \eqref{EiHi} with incident direction $\bd$ by a perfectly conducting plane $\G:=\{\bx\in\R^3: \bx\cdot\bm{\nu} = 0\}$ with normal vector $\bm{\nu}$ that contains the origin. In this case, the scattered field is given explicitly by
\ben
\bE^s (\bx, \bd, \bp, k) = -\bp^s e^{ik\bx\cdot \bd^s},
\quad
\bH^s (\bx, \bd, \bp, k) = -(\bd^s\times\bp^s) e^{ik\bx\cdot\bd^s},
\enn
where $\bd^s = \bd -2 \, ( \bm{\nu}\cdot\bd) \,  \bm{\nu}$ and $\bp^s = \bp -2 \, (\bm{\nu}\cdot\bp) \,  \bm{\nu}$. Indeed, straightforward calculations show that
\ben
\bm{\nu}\times (\bd^s\times\bp^s) = -\bm{\nu}\times(\bd\times\bp)
\qquad\text{and}\qquad
\bx\cdot \bd = \bx\cdot \bd^s \, ,  \quad \bx\in \G,
\enn
and therefore
\ben
\bm{\nu}\times(\bE^i+\bE^s) =0 \quad\mbox{and}\quad \bm{\nu}\times (\bH^i+\bH^s) = 2 \bm{\nu}\times \bH^i \qquad \text{on } \Gamma.
\enn
In the physical optics approximation, one assumes that the wavelength is significantly smaller than the size of the obstacle. Thus, the boundary of the obstacle $D$ locally may be considered at each point $\bx\in\pa D$ as a plane with normal $\bm{\nu}(\bx)$. The corresponding approximated scattered fields will be denoted by $(\bE^s_{\text{po}},\bH^s_{\text{po}})$ with similar notation for total fields, far field patterns, etc. Introducing the illuminated region $\pa D_{-}(\bd):=\{\bx\in\pa D |\, \bm{\nu}(\bx)\cdot \bd<0\}$ and the shadow region $\pa D_{+}(\bd):=\{\bx\in\pa D |\, \bm{\nu}(\bx)\cdot \bd\geq0\}$, respectively, for a plane wave in the incident direction $\bd$, this leads to setting
\be\label{Kirchhoff-pec}
\bm{\nu}\times \bH_{\text{po}}  = \left\{
                            \begin{array}{ll}
                              2\bm{\nu}\times \bH^i  , & \hbox{on $\pa D_{-}(\bd)$;} \\
                              0, & \hbox{on $\pa D_{+}(\bd)$.}
                            \end{array}
                          \right.
\en

From here on, throughout the paper, we will assume that $k_{min}>0$ is large enough such that the physical optics approximation $(\bE_{\text{po}},\bH_{\text{po}})$ to $(\bE,\bH)$ is accurate.
Inserting \eqref{Kirchhoff-pec} into \eqref{HuygenPrinciple}, with the help of the representation \eqref{EiHi} of the plane waves, we deduce that
\be\label{EinftyKirchhoff}
&&\bE^{\infty}_{\text{po}}({\bf\hx}, \bd, \bp, k)\cr
&=&2ik \, {\bf\hx}\times\int_{\pa D_{-}(\bd)} [\bm{\nu}(\by)\times \bH^i(\by)]\times {\bf\hx} \, e^{-ik{\bf\hx}\cdot\by}ds(\by)\cr
&=&2ik \, {\bf\hx}\times\int_{\pa D_{-}(\bd)} [\bm{\nu}(\by)\times (\bd\times\bp)]\times {\bf\hx} \,  e^{ik(\bd-{\bf\hx})\cdot\by}ds(\by), \quad {\bf\hx}, \bd, \bp\in\bbS, k\in K.\quad
\en
The implies that the shadow region $\pa D_{+}(\bd)$ gives no contribution to the far field pattern $\bE^{\infty}_{\text{po}}({\bf\hx}, \bd, \bp, k)$. Thus, it is impossible to determine the shadow region $\pa D_{+}(\bd)$ at high frequencies using the far field pattern $\bE^{\infty}_{\text{po}}({\bf\hx}, \bd, \bp, k)$.

\section{Uniqueness}\label{Sec-Uni}

We consider the following generalized backscattering electrical far field patterns over some band of frequencies
\ben
\bE^{\infty}_{\text{po}}({\bf\hx},\bd,\bp,k) \quad\mbox{and}\quad \bE^{\infty}_{\text{po}}(-{\bf\hx},-\bd,\bp,k),  \quad k\in K.
\enn
for one fixed observation direction ${\bf\hx}$, one incident direction $\bd$ and one polarization $\bp$. In particular, this reduces to the classical backscattering data if ${\bf\hx} = -\bd$.

For any fixed direction $\bbtheta\in\Theta_l$, the $\bbtheta$-strip hull of $D$ is defined by
\be\label{S_D_def}
S_{D}(\bbtheta):=  \{ \by\in \mathbb \R^{3}\; | \; \inf_{z\in D}z\cdot \bbtheta \leq y\cdot \bbtheta\leq \sup_{z\in D}z\cdot \bbtheta\},
\en
which is the smallest strip (region between two parallel hyper-planes) containing $\ov{D}$ with normals $\pm \bbtheta$.

\begin{figure}[htbp]
\centering
\includegraphics[width=3in]{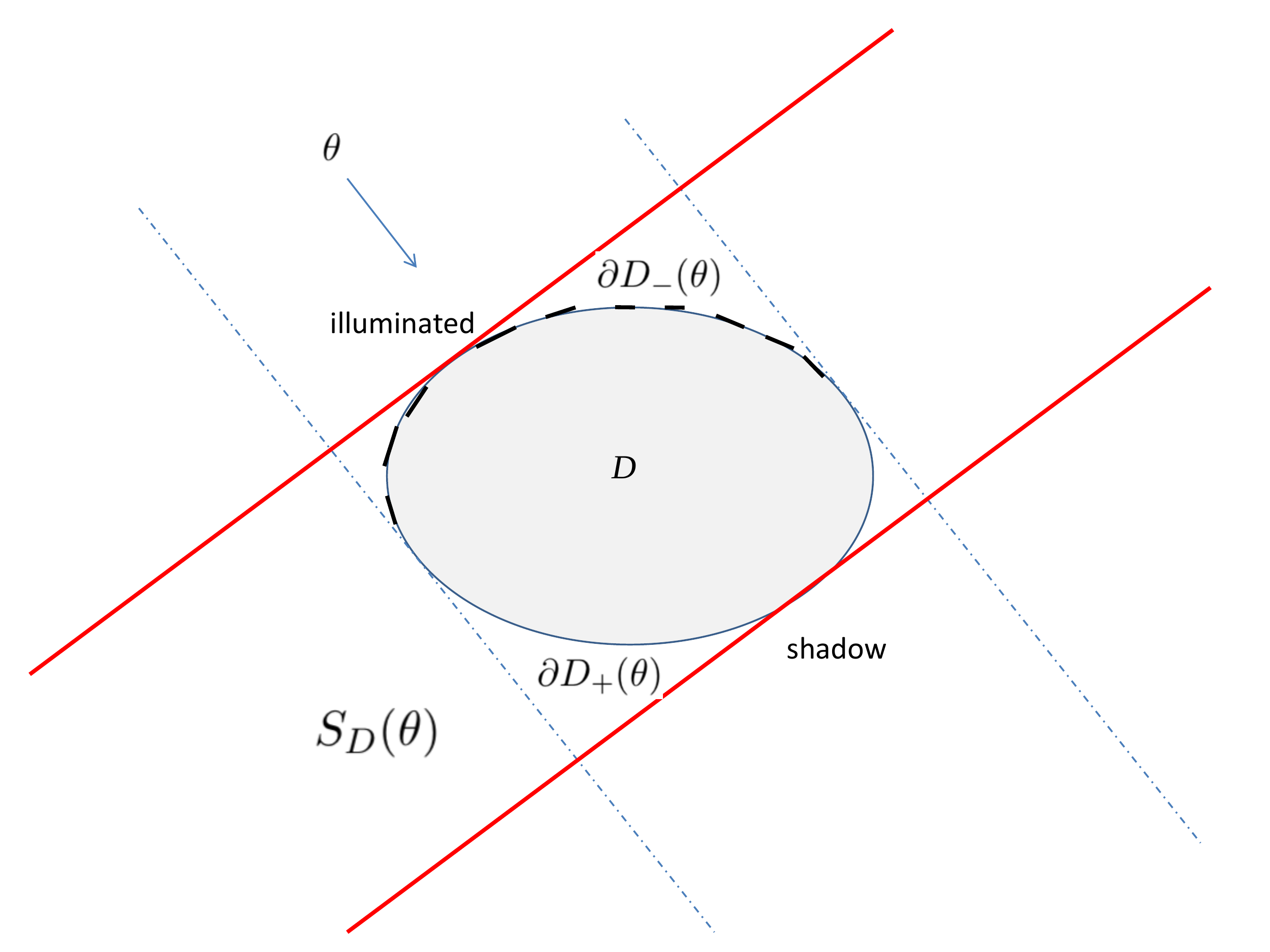}
\caption{The $\bbtheta$-strip $S_{D}(\bbtheta)$, illuminated part $\pa D_{-}(\bbtheta)$ and shadow region $\pa D_{+}(\bbtheta)$.}
\label{IlluminatedShadow}
\end{figure}

\begin{theorem}\label{Strip-uniqueness}
For one fixed incident direction $\bd\in \bbS$, one fixed polarization $\bp\in \bbS$ and one fixed observation direction $\bf\hx\in\bbS$ such that
\be\label{Aassumption}
A({\bf\hx}, \bd, \bp):={\bf\hx}\times \big( [(\bd-\bf\hx)\times(\bd\times\bp)]\times {\bf\hx} \big) \neq {\bf 0} \, ,
\en
define
\be\label{phi}
\phi:=\frac{\bd-\bf\hx}{|\bd-\bf\hx|}.
\en
Then the corresponding $\phi$-strip $S_D(\phi)$ is uniquely determined by the generalized backscattering electric far field patterns $\bE^{\infty}_{\text{po}}({\bf\hx},\bd,\bp,k)$ and $\bE^{\infty}_{\text{po}}(-{\bf\hx},-\bd,\bp,k)$ for all $k\in K$.
\end{theorem}

\begin{proof}
We recall the electric far field pattern representation in the physical optics approximation \eqref{EinftyKirchhoff} and observe that
\ben
&&\bE^{\infty}_{\text{po}}(-{\bf\hx}, -\bd, \bp, k)\cr
&=& -2ik{\bf\hx}\times\int_{\pa D_{+}(\bd)} [\bm{\nu}(\by)\times (\bd\times\bp)]\times {\bf\hx} \, e^{-ik(\bd-{\bf\hx})\cdot\by}ds(\by), \quad {\bf\hx}, \bd, \bp\in\bbS, k\in K.\quad
\enn
Combining this with \eqref{EinftyKirchhoff}, we deduce that
\be\label{Uinfty}
&&U^{\infty}({\bf\hx}, \bd, \bp, k) \nonumber \\[0.75ex]
&:=&\bE^{\infty}_{\text{po}}({\bf\hx},\bd,\bp,k) + \ov{\bE^{\infty}_{\text{po}}(-{\bf\hx}, -\bd, \bp, k)}\cr
&=& 2ik{\bf\hx}\times\int_{\pa D} [\bm{\nu}(\by)\times (\bd\times\bp)]\times {\bf\hx} \, e^{ik(\bd-{\bf\hx})\cdot\by}ds(\by)\cr
&=& 2ik{\bf\hx}\times \left( \int_{D} \curl_\by [(\bd\times\bp) e^{ik(\bd-{\bf\hx})\cdot\by}]d(\by)\times {\bf\hx} \right) \cr
&=& -2k^2 \, A({\bf\hx}, \bd, \bp) \int_{D} e^{ik(\bd-{\bf\hx})\cdot\by}d(\by), \quad {\bf\hx}, \bd, \bp\in\bbS, k\in K.\quad
\en

For the two fixed directions ${\bf\hx}$, $\bd \in \Theta_l$, we define
\ben
\Pi_{\alpha}:=\{\by\in \R^3 | (\bd-{\bf\hx})\cdot\by+\alpha=0\}
\enn
to be a hyperplane with normal $\phi$ given by \eqref{phi}.
Then we have
\be\label{UinfFourier}
U^{\infty}({\bf\hx}, \bd, \bp, k) = -2k^2A({\bf\hx}, \bd, \bp) \int_{\R}\hat{\chi}(\alpha) e^{-ik\alpha}d\alpha, \quad {\bf\hx}, \bd, \bp\in\bbS, k\in K.\quad
\en
where
\be
\hat{\chi}(\alpha):=\int_{\Pi_\alpha}\chi(y)ds(y) \, ,
\en
and $\chi$ denotes the characteristic function of $D$. As $D$ is compact, the right hand side of \eqref{UinfFourier} is an analytic function with respect to $k$ for all $k \in \R$.  Thus, by analyticity, we can extend right hand side of \eqref{UinfFourier} and hence the data $U^{\infty}({\bf\hx}, \bd, \bp, k)$ analytically to all $k\in\R$.

Applying the Fourier transform, the equality \eqref{UinfFourier} implies that $\hat{\chi}$ can be uniquely determined by $U^{\infty}({\bf\hx}, \bd, \bp, k)$ and $A({\bf\hx}, \bd, \bp)$.
Note that $S_D(\phi)$ defined by \eqref{S_D_def} satisfies
\ben
S_D(\phi)=\ov{\bigcup_{\alpha\in \R} \{\Pi_{\alpha}| \hat{\chi}(\alpha)\neq 0\}},
\enn
which implies that the strip $S_D(\phi)$ is uniquely determined by $\hat{\chi}$, and thus by the generalized backscattering data $U^{\infty}({\bf\hx}, \bd, \bp, k)$ for all $k\in K$ and three fixed directions $\bd$, $\bp$, ${\bf\hx} \in \Theta_l$. The proof is complete.
\end{proof}

The assumption $A({\bf\hx}, \bd, \bp)\neq {\bf0}$ implies that $\bf\hx\neq \bd$, thus the unit vector $\phi$ given in \eqref{phi} is well defined. Actually, for the classical backscattering case, i.e., $\bf\hx = -\bd$, we have
\be\label{Aback}
A({\bf\hx}, \bd, \bp)= {\bf\hx}\times \left( [(\bd-{\bf\hx})\times(\bd\times\bp)]\times {\bf\hx} \right) = -2\bp
\en
by using the fact that $\bd\cdot\bp = 0$. Clearly, $A({\bf\hx}, \bd, \bp)\neq \bf 0$ in this case.
Furthermore, as indicated in the following corollary, less data are needed to uniquely determine the strip.\\

\begin{corollary}
For one fixed incident direction $\bd\in \bbS$, one fixed polarization $\bp\in \bbS$ satisfying $\bp\cdot\bd = 0$, we choose a unit vector ${\bf e}\in \bbS$ such that ${\bf e}\cdot\bp\neq 0$. Then the $\bd$-strip $S_D(\bd)$ is uniquely determined by the backscattering data ${\bf e}\cdot\bE^{\infty}_{\text{po}}(-\bd,\bd,\bp,k)$ and ${\bf e}\cdot\bE^{\infty}_{\text{po}}(\bd,-\bd,\bp,k)$ for all $k\in K$.
\end{corollary} \\

Note that to determine the $\phi$-strip $S_D(\phi)$, we actually used one pair of directions $(\bf\hx, \bd)$ and $(-\bf\hx, -\bd)$.
Define
\be\label{phij}
\phi_j:=\frac{\bd_j-\bf\hx_j}{|\bd_j-\bf\hx_j|}, \quad j =1, 2, 3.
\en
As a corollary of Theorem \ref{Strip-uniqueness} we immediately have a uniqueness result for obstacle support with at most three pairs of opposite observation directions.\\

\begin{corollary}
Let $({\bf\hx}_j, \bd_j)$, \, $j=1, 2, 3$ be three pairs of directions such that the assumption \eqref{Aassumption} holds and the corresponding three unit directions $\phi_j$ given by \eqref{phij} are linearly independent. Take $\bp_j$ such that $\bp_j\cdot\bd_j =0$. Then the smallest parallel hexahedron containing the obstacle with normals $\pm\phi_j$ can be uniquely determined by the generalized backscattering data
$\bE^{\infty}_{\text{po}}({\bf\hx}_j,\bd_j,\bp_j,k)$ and $\bE^{\infty}_{\text{po}}(-{\bf\hx}_j,-\bd_j,\bp_j,k)$, $j=1$, $2$, $3$, for all $k\in K$.
\end{corollary}

\section{A simple sampling method}\label{Sec-DSM}

We consider the following indicator

\be\label{indicator}
I(\bz) := \sum_{\bd\in\Theta_l}I_{\bd}(\bz):=\sum_{\bd\in\Theta_l}\left|\int_{K}\frac{1}{4k^2}\bp\cdot U^{\infty}({\bf\hx}, \bd, \bp, k)e^{-ik(\bd-{\bf\hx})\cdot\bz}dk\right|^2, \quad \bz\in\R^3 \, , \qquad
\en
where $U^{\infty}({\bf\hx}, \bd, \bp, k)=\bE^{\infty}({\bf\hx},\bd,\bp,k) + \ov{\bE^{\infty}(-{\bf\hx}, -\bd, \bp, k)}$.
Of particular interest is the backscattering case, i.e., $\bf\hx=-\bd$. Then we can combine \eqref{Uinfty} and \eqref{Aback} to obtain
\ben
U^{\infty}(-\bd, \bd, \bp, k) = 4 k^2 \bp \int_{D} e^{2ik\bd \cdot\by}d(\by).
\enn
Inserting this into \eqref{indicator} and interchanging the order of integration we observe that
\be\label{Ibd}
I_{\bd}(\bz)
&=&\left|\int_{K}\int_{D} e^{2ik\bd \cdot\by}d(\by)e^{-2ik\bd\cdot\bz}dk\right|^2\cr
&=&\left|\int_{D} \int_{K}e^{2ik\bd \cdot(\by-\bz)}dk d(\by)\right|^2, \quad \bz\in\R^3 \, .
\en
It is obvious that
\ben
I_{\bd}(\bz + \alpha \bd^{\perp}) = I_{\bd}(\bz)
\enn
for any $\alpha\in\R$ and any vector $\bd^{\perp}$ satisfying $\bd\cdot \bd^{\perp}=0$.
We rewrite the representation \eqref{Ibd} in the form
\ben
I_{\bd}(\bz)=\left|\int_{D} \int_{K}e^{2ik\bd \cdot(\by-\bz)}dk \, d(\by)\right|^2, \quad \bz\in\R^3.
\enn
In particular, for $\bz\in\R^3\ba\ov{S_{D}(\bd)}$, we find
\ben
I_{\bd}(\bz)=\left|\int_{D}\frac{e^{2ik\bd \cdot(\by-\bz)}|_{K}}{2\bd \cdot(\by-\bz)} d(\by)\right|^2, \quad \bz\in\R^3\ba\ov{S_{D}(\bd)},
\enn
where the numerator $e^{2ik\bd \cdot(\by-\bz)}|_{K}:=e^{2ik_{max}\bd \cdot(\by-\bz)}- e^{2ik_{min}\bd \cdot(\by-\bz)}$ is clearly bounded. Thus it is expected that
the indicator $I_{\bd}(\bz)$ decays like $1/|\bd \cdot(\by-\bz)|^2$ as the sampling point $\bz$ moves away from the strip $S_{D}(\bd)$.

\section{Numerical examples}\label{Sec-Num}

In this section, we present a variety of numerical examples in three dimensions to illustrate the applicability, effectiveness and robustness of our sampling methods with broadband sparse data.

The data for the numerical experiments were generated using the boundary element library \texttt{bempp} (\url{https://bempp.com} \cite{SCROGG2017, SmiArrBet2015}). The boundary integral equation to be solved is derived as in \cite{ConDem2002, SCROGG2017} including a stabilization applied to the electric field operator. We repeat the main points here for self-consistency of the paper. We use the electric and magnetic potentials, defined for $\bm{\varphi} \in H^{-1/2}( \DDiv, \partial D)$ (see \cite{BufCos2002} for a precise definition of this and other relevant trace spaces),
\begin{align*}
  \EP \bm{\varphi}(\bx) & = \I k \, \int_{\partial D} \Phi(\bx,\by) \, \bm{\varphi}(\by) \, \D s(\by)  - \frac{1}{\I k} \Dgrad \int_{\partial D} \Phi(\bx,\by) \, \DDiv_{\partial D} \bm{\varphi}(\by) \, \D s(\by) \, , \\
  \MP \bm{\varphi}(\bx) & = \Dcurl \int_{\partial D} \Phi(\bx,\by) \, \bm{\varphi}(\by) \, \D s(\by) \, ,
\end{align*}
and the traces
\[
  \gamma_t \bE = \bE \times \bm{\nu} \, , \qquad \gamma_N \bE = \frac{1}{\I k} \, \gamma_t \Dcurl \bE \qquad \text{on } \partial D \, .
\]
Superscripts $\cdot^\pm$ will be used to indicate whether these traces are taken from outside or inside $\partial D$, respectively. If no superscript is given, a trace from outside $\partial D$ is understood.

In this notation, the Stratton-Chu formula for a radiating electric field outside of $D$ is
\[
  \bE^s = - \MP \gamma_t \bE^s - \EP \gamma_N \bE^s \, .
\]
Applying the trace operators to the potentials and averaging defines the electromagnetic boundary integral operators
\[
   \mathcal{E} = \frac{1}{2} \left( \gamma_t^+ \EP + \gamma_t^- \EP \right) , \qquad
   \mathcal{H} = \frac{1}{2} \left( \gamma_t^+ \MP + \gamma_t^- \MP \right) .
\]
From the jump relations we obtain the equations
\[
  \gamma_t \EP = \mathcal{E} \, , \quad
  \gamma_N \EP = - \frac{1}{2} \, \mathcal{I} + \mathcal{H} \, , \quad
  \gamma_t \MP = - \frac{1}{2} \, \mathcal{I} + \mathcal{H} \, , \quad
  \gamma_N \MP = - \mathcal{E}
\]
 for the exterior traces. Applying the trace operators to the Stratton-Chu formula gives the exterior Calderon projector,
\[
  \mathcal{C}^+ \begin{pmatrix} \gamma_t \bE^s \\ \gamma_N \bE^s  \end{pmatrix}
  =  \begin{pmatrix} \gamma_t \bE^s \\ \gamma_N \bE^s  \end{pmatrix}
 \qquad \text{with} \qquad
 \mathcal{C}^+
 = \frac{1}{2} \, \mathcal{I} - \mathcal{A}
 = \frac{1}{2} \, \mathcal{I} - \begin{pmatrix} \mathcal{H}  & \mathcal{E} \\ -\mathcal{E} & \mathcal{H} \end{pmatrix} .
\]
$\mathcal{A}$ is also called the multitrace operator.

The two rows of the Calderon projector equation with $\gamma_N \bE^s$ as the unknown give the electric and magnetic field boundary integral equations (EFIE and MFIE), respectively. We apply a regularizing operator $\mathcal{R}$ to the EFIE, which is  the operator $\mathcal{E}$ with $k$ replaced by $\I k$, and add both equations to obtain the combined field boundary integral equation (CFIE)
\[
   \left( \frac{1}{2} \, \mathcal{I} + \mathcal{H} - \mathcal{R E} \right) \gamma_N \bE^s = - \left( \mathcal{R} \left[ \frac{1}{2} \, \mathcal{I} + \mathcal{H} \right] + \mathcal{E} \right) \gamma_t \bE^i \, .
\]

\begin{center}
  \begin{tabular}{c@{\qquad\qquad}c}
    \includegraphics[width=0.425\linewidth]{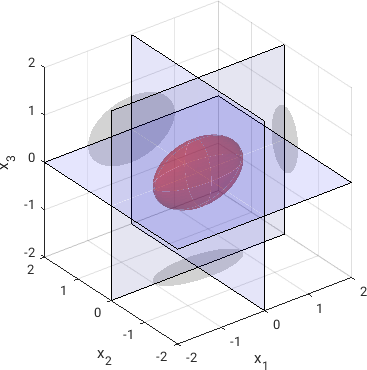}
    & \includegraphics[width=0.425\linewidth]{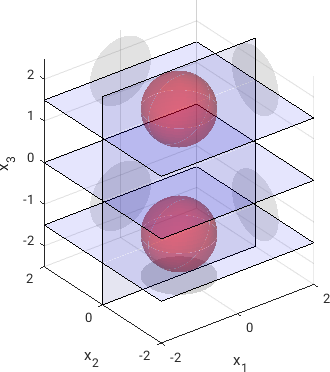} \\[1ex]
    (a) & (b)
  \end{tabular}
  \captionof{figure}{Obstacles used in the numerical experiments together with planes, on which the indicator is plotted: (a) ellipsoid with planes $x_1 = 0$, $x_2 = 0$, $x_3 = 0$, (b) two spheres with planes $x_3 = -1.5$, $x_3 = 0$, $x_3 = 1.5$, $x_2 = 0$.}
  \label{fig:settings}
\end{center}

\bigskip

This is the boundary integral equation solved using \texttt{bempp}. The scattered field is then obtained from Stratton-Chu as
\[
 \bE^s = \MP \gamma_t \bE^i - \EP \gamma_N \bE^s \, ,
\]
and the far fields are obtained from the corresponding asymptotic expansion.

We will consider two settings for scattering problems, one for a convex and one for a non-convex obstacle. In the first setting, the obstacle is an ellipsoid centered at the origin with half axis $1.0$, $0.4$ and $0.7$, respectively. In the second setting, the obstacle is made up of two perfectly conducting balls of radius $0.8$ with center at $(0,0,\pm 1.5)$, respectively. Both settings, together with the planes in which the indicator function will be plotted subsequently, are displayed in Figure \ref{fig:settings}.

In both settings, incident plane waves as in \eqref{EiHi} for wave number $k \in [10,20]$ are considered. This corresponds to wave lengths ranging from 0.6281 down to 0.3142. For the evaluation of the indicator function $I_{\bd}$ given in \eqref{indicator} we approximate the integral by the composite trapezoidal rule with 40 uniformly spaced wave numbers in this interval. From the point of view that high frequency asymptotics are the basis of the derivation of our method, it is perhaps surprising that we use such moderately sized wave numbers. However, the quality of the reconstructions is surprisingly good, as we will show below.

The boundaries of both obstacles were discretized using a triangulation with mesh size $h = 0.05$. This corresponds to $6136$ elements for the ellipsoid and $16686$ elements for the two spheres. The boundary element spaces used for solving the integral equations are spanned by Rao-Wilton-Glisson basis functions \cite{RWG} and have $9204$ and $25029$ degrees of freedom respectively. Using H-matrix compression, about 16 Gb of memory was required for solving the scattering problems for the ellipsoid and 38 Gb for the two sphere obstacle. The linear systems were solved using GMRES which always converged in 16--18 iterations independent of the incident direction and the wave number.

\begin{center}
  \begin{tabular}{c@{\quad}c@{\quad}c}
    \includegraphics[width=0.31\linewidth]{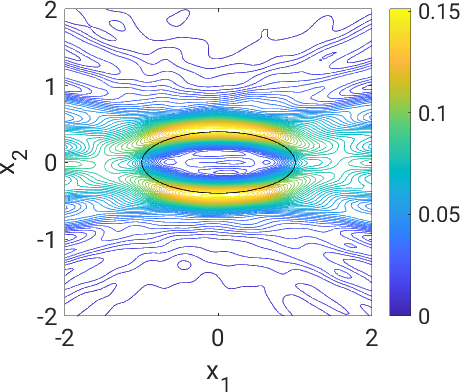}
    & \includegraphics[width=0.31\linewidth]{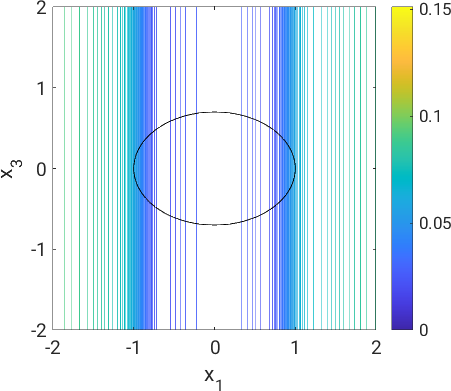}
    & \includegraphics[width=0.31\linewidth]{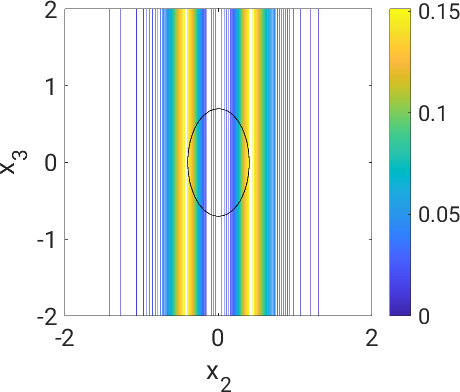} \\[1ex]
    (a) & (b) & (c)
  \end{tabular}
  \vspace{-1ex}

 \captionof{figure}{Plot of $I(\bz)$, ellipsoid obstacle, 40 incident directions in $(x_1,x_2)$-plane; (a) $x_3 = 0$; (b) $x_2 = 0$; (c) $x_1 = 0$.}
 \label{fig:ellipsoid_xy}
\end{center}

\bigskip

\begin{center}
  \begin{tabular}{c@{\quad}c@{\quad}c}
    \includegraphics[width=0.31\linewidth]{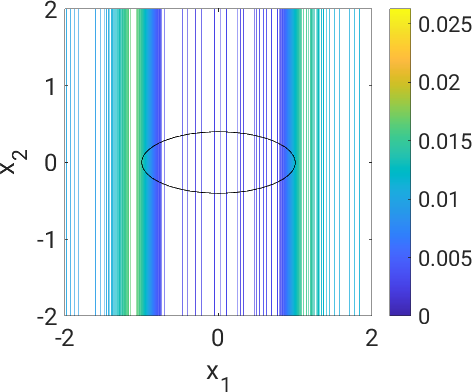}
    & \includegraphics[width=0.31\linewidth]{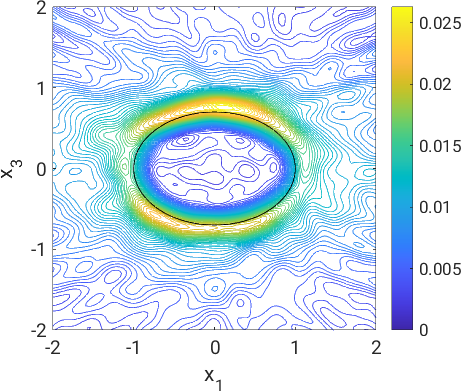}
    & \includegraphics[width=0.31\linewidth]{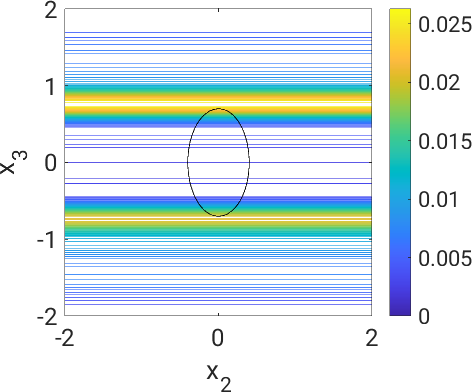} \\[1ex]
    (a) & (b) & (c)
  \end{tabular}
  \vspace{-1ex}

 \captionof{figure}{Plot of $I(\bz)$, ellipsoid obstacle, 40 incident directions in $(x_1,x_3)$-plane; (a) $x_3 = 0$; (b) $x_2 = 0$; (c) $x_1 = 0$.}
 \label{fig:ellipsoid_xz}
\end{center}

\bigskip

\begin{center}
  \begin{tabular}{c@{\quad}c@{\quad}c}
    \includegraphics[width=0.31\linewidth]{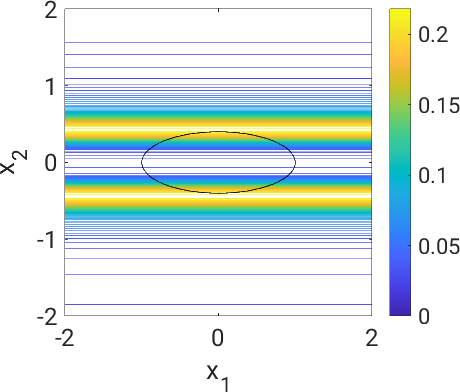}
    & \includegraphics[width=0.31\linewidth]{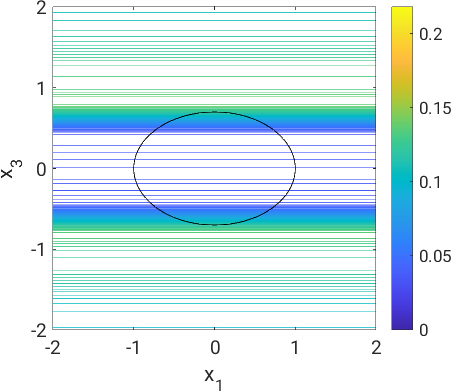}
    & \includegraphics[width=0.31\linewidth]{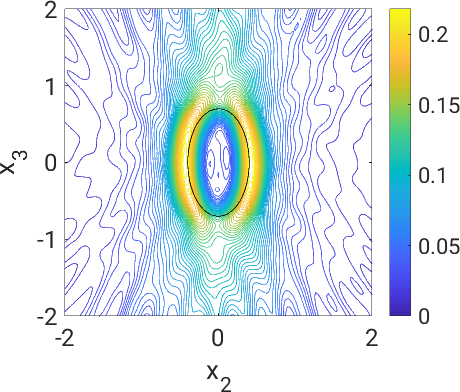} \\[1ex]
    (a) & (b) & (c)
  \end{tabular}
  \vspace{-1ex}

 \captionof{figure}{Plot of $I(\bz)$, ellipsoid obstacle, 40 incident directions in $(x_2,x_3)$-plane; (a) $x_3 = 0$; (b) $x_2 = 0$; (c) $x_1 = 0$.}
 \label{fig:ellipsoid_yz}
\end{center}

\bigskip
\bigskip

For the plots of the indicator function, we only present the backscattering case. Before computing the indicator function, random noise was added to the data. For each scattering problem, the maximum amplitude of the corresponding far field pattern was computed. The data for this scattering problem was then perturbed by a uniformly distributed random complex vector of Euclidean norm at most $10\%$ of this maximum amplitude.

\newpage

\begin{center}
  \begin{tabular}{c@{\qquad}c}
    \includegraphics[width=0.32\linewidth]{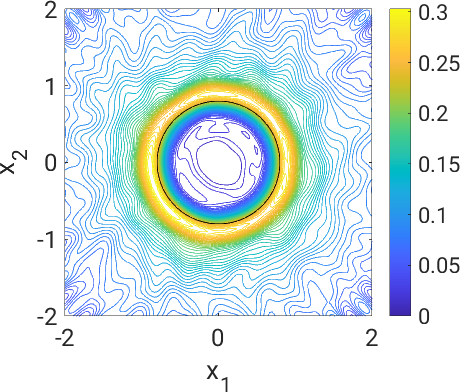}
    & \includegraphics[width=0.32\linewidth]{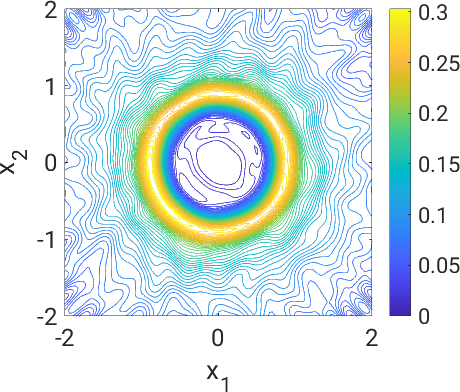} \\
    (a) & (b) \\[1ex]
    \raisebox{6ex}{\includegraphics[width=0.32\linewidth]{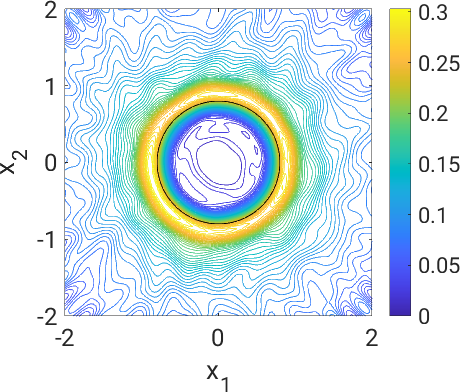}}
    & \includegraphics[width=0.32\linewidth]{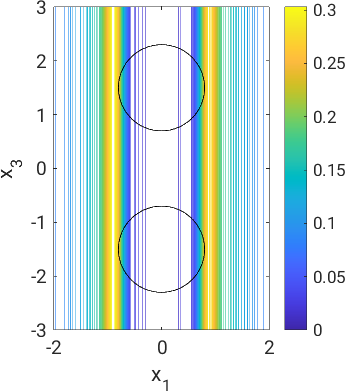} \\
    (c) & (d)
  \end{tabular}
  \vspace{-2ex}

 \captionof{figure}{Plot of $I(\bz)$, two spheres obstacle, 40 incident directions in $(x_1,x_2)$-plane; (a) $x_3 = -1.5$; (b) $x_3 = 0$; (c) $x_3 = 1.5$; (d) $x_2 = 0$.}
 \label{fig:twospheres_xy}
\end{center}

\bigskip
\bigskip

In the first set of experiments, the directions of incidence lie in one of the coordinate planes. Our theory makes no special assumption on the polarization other than that in the definition of $U^\infty$, the polarization for the opposite directions must be identical. In fact, \eqref{Ibd} shows that the indicator $I_{\bd}$ is independent of the choice of $\bp$. In the results presented here, $\bp$ is chosen differently for each direction of incidence. Other experiments which were conducted indeed show that there is no visible difference in the indicator function when a different choice for $\bp$ is made, for example the normal of the plane in which the directions of incidence lie.

Results for the ellipsoid are displayed in Figures \ref{fig:ellipsoid_xy} -- \ref{fig:ellipsoid_yz} and for the two sphere obstacle in Figures \ref{fig:twospheres_xy} and \ref{fig:twospheres_xz}. The boundary of the obstacle is always indicated as a black line. In each case, 40 incident directions are used, distributed uniformly on the unit circle in the particular coordinate plane. We have not included the plot for the two spheres corresponding to Figure \ref{fig:ellipsoid_yz} as due to symmetry of the object and the orientation of the planes on which the indicator function is plotted, this plot adds nothing to what is displayed in Figure \ref{fig:twospheres_xz}. In each case we observe highest values of the indicator along relatively flat boundary parts that correspond with the boundary of the convex hull of the obstacle. Also, the highest values of the indicator function are located just outside the obstacle.

This effect is due to the superposition of the results for several directions. In Figure \ref{fig:ellipsoid_onedir} we show plots of $I_{\bd}(\bz)$, i.e. using just one direction $\bd$ (and its opposite). Here, the boundary of the strip between maximal values of $I_{\bd}(\bz)$ touches the obstacle.

\begin{center}
  \begin{tabular}{c@{\qquad}c}
    \includegraphics[width=0.32\linewidth]{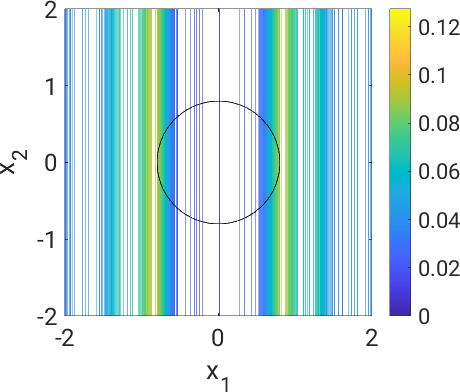}
    & \includegraphics[width=0.32\linewidth]{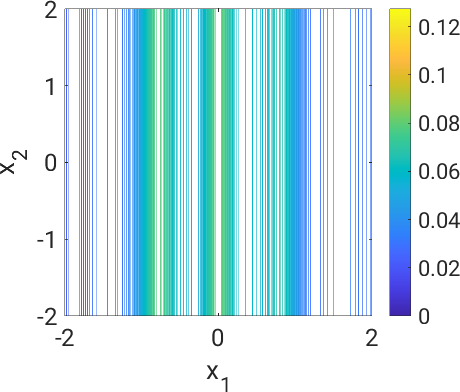} \\
    (a) & (b) \\[1ex]
    \raisebox{5ex}{\includegraphics[width=0.32\linewidth]{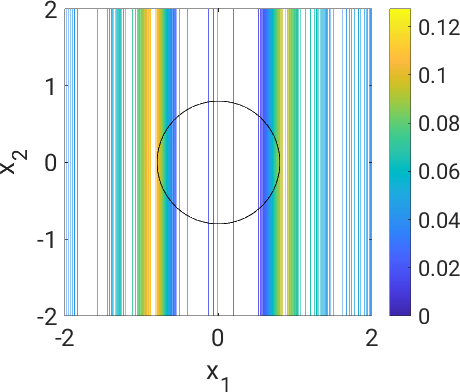}}
    & \includegraphics[width=0.32\linewidth]{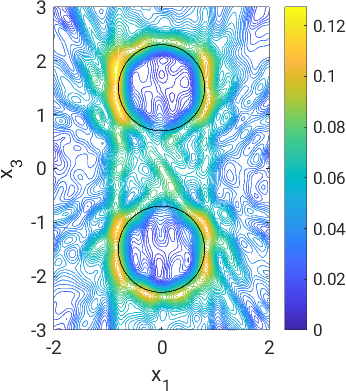} \\
    (c) & (d)
  \end{tabular}
  \vspace{-2ex}

 \captionof{figure}{Plot of $I(\bz)$, two spheres obstacle, 40 incident directions in $(x_1,x_3)$-plane; (a) $x_3 = -1.5$; (b) $x_3 = 0$; (c) $x_3 = 1.5$; (d) $x_2 = 0$.}
 \label{fig:twospheres_xz}
\end{center}

\bigskip\bigskip

\begin{center}
  \begin{tabular}{c@{\quad}c@{\quad}c}
    \includegraphics[width=0.31\linewidth]{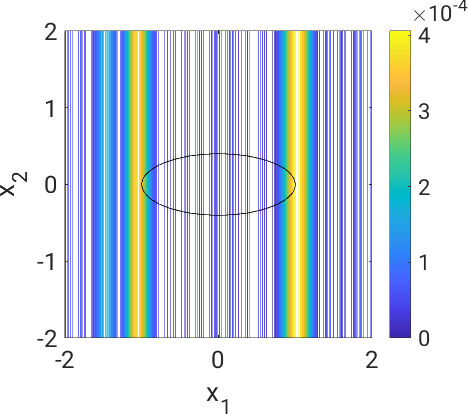}
    & \includegraphics[width=0.31\linewidth]{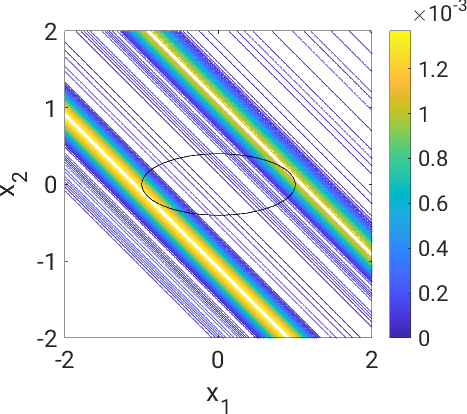}
    & \includegraphics[width=0.31\linewidth]{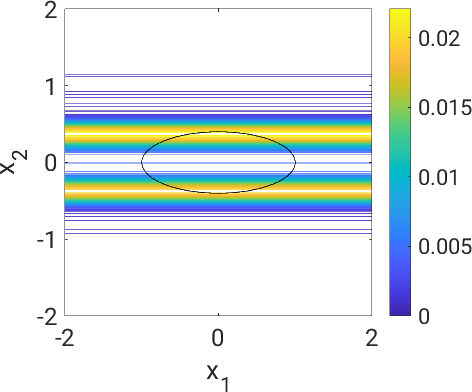} \\[1ex]
    (a) & (b) & (c)
  \end{tabular}
  \vspace{-1ex}

 \captionof{figure}{Plot of $I_{\bd}(\bz)$, ellipsoid obstacle, 1 incident direction; (a) $\bd = (1,0,0)^\top$; (b) $\bd = (1,1,0)^\top / \sqrt{2}$; (c) $\bd = (0,1,0)^\top$.}
 \label{fig:ellipsoid_onedir}
\end{center}

\bigskip
\bigskip

The plots in Figure \ref{fig:ellipsoid_onedir} also show that for different directions of incidence and observation the maximal values of the indicator function $I_{\bd}$ can be quite of quite diverse magnitude. This observation let us to experiment with variations of the indicator function in which contributions for different directions are scaled differently and are hence more balanced overall. Define
\[
  I^p (\bz) = \sum_{\bd\in\Theta_l} \frac{ I_{\bd}(\bz) }{ \left\| I_{\bd} \right\|_\infty^p} \, , \qquad \bz\in\R^3 \, , \quad p \in \left\lbrace 0, \frac{1}{2}, 1 \right\rbrace .
\]

\begin{center}
  \begin{tabular}{c@{\quad}c@{\quad}c}
    \includegraphics[width=0.31\linewidth]{plots/ellipsoid_mxwl_in_plane_xy_z_0.png}
    & \includegraphics[width=0.31\linewidth]{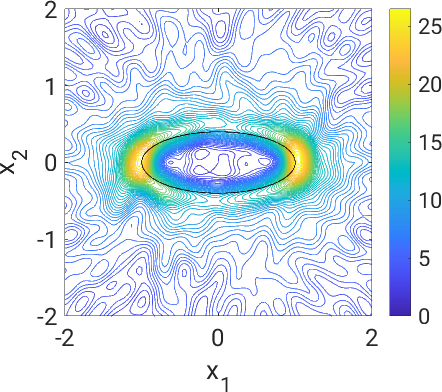}
    & \includegraphics[width=0.31\linewidth]{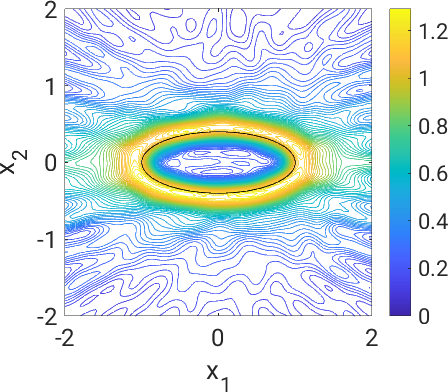} \\[1ex]
    (a) $I^0(\bz)$ & (b) $I^1(\bz)$ & (c) $I^{1/2}(\bz)$
  \end{tabular}
  \vspace{-1ex}

   \captionof{figure}{The plot from Figure \ref{fig:ellipsoid_xy} (a) with different indicators.}
  \label{fig:ellipsoid_xy_compare}
\end{center}

\bigskip

\begin{center}
  \begin{tabular}{c@{\quad}c@{\quad}c}
    \includegraphics[width=0.31\linewidth]{plots/ellipsoid_mxwl_in_plane_xz_y_0.png}
    & \includegraphics[width=0.31\linewidth]{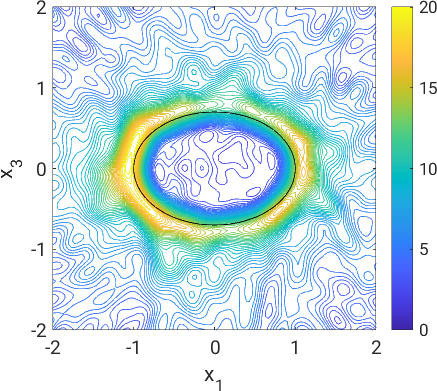}
    & \includegraphics[width=0.31\linewidth]{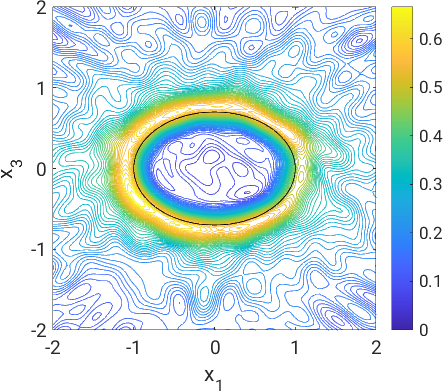} \\[1ex]
    (a) $I^0(\bz)$ & (b) $I^1(\bz)$ & (c) $I^{1/2}(\bz)$
  \end{tabular}
  \vspace{-1ex}

   \captionof{figure}{The plot from Figure \ref{fig:ellipsoid_xz} (b) with different indicators.}
  \label{fig:ellipsoid_xz_compare}
\end{center}

\bigskip

\begin{center}
  \begin{tabular}{c@{\quad}c@{\quad}c}
    \includegraphics[width=0.31\linewidth]{plots/ellipsoid_mxwl_in_plane_yz_x_0.png}
    & \includegraphics[width=0.31\linewidth]{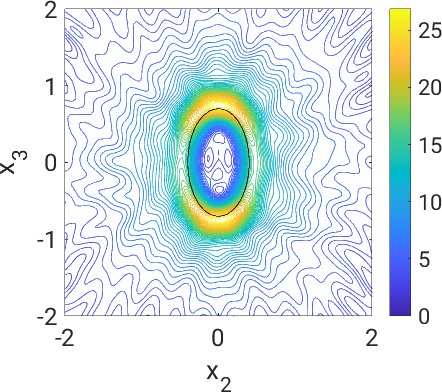}
    & \includegraphics[width=0.31\linewidth]{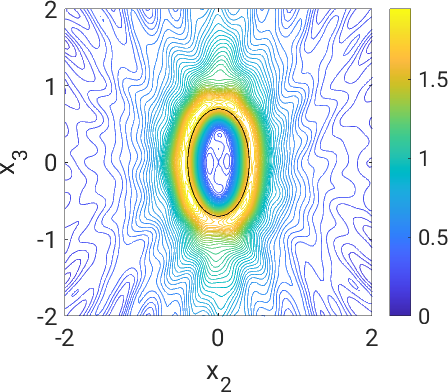} \\[1ex]
    (a) $I^0(\bz)$ & (b) $I^1(\bz)$ & (c) $I^{1/2}(\bz)$
  \end{tabular}
  \vspace{-1ex}

   \captionof{figure}{The plot from Figure \ref{fig:ellipsoid_yz} (c) with different indicators.}
  \label{fig:ellipsoid_yz_compare}
\end{center}

\bigskip

The choice $p = 0$ corresponds to our original indicator. We compare the effect of different scalings for the case of the ellipsoid obstacle in Figures \ref{fig:ellipsoid_xy_compare} -- \ref{fig:ellipsoid_yz_compare}. It can be clearly seen that the plots of $I^0$ and $I^1$ complement each other while $I^{1/2}$ gives the best overall result. For the non-convex obstacle consisting of two spheres we also obtain a much improved reconstruction of the complete boundary using $I^{1/2}$ compared to $I^0$, as can be seen in Figure \ref{fig:twospheres_xz_scaled}.

\newpage

\begin{minipage}{0.38\linewidth}
   \begin{center}
    \includegraphics[width=0.84\linewidth]{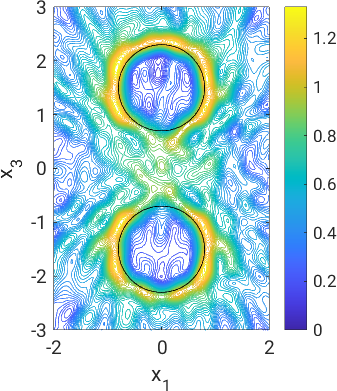}

    \captionof{figure}{Figure \ref{fig:twospheres_xz} (d) using $I^{1/2}$.}
    \label{fig:twospheres_xz_scaled}
   \end{center}
\end{minipage}
\hfill
\begin{minipage}{0.38\linewidth}
   \begin{center}
    \includegraphics[width=0.95\linewidth]{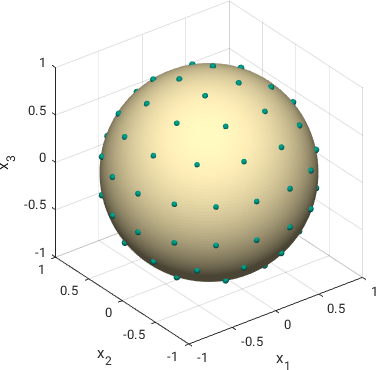}

    \captionof{figure}{80 uniformly distributed points on the unit sphere used as directions of incidence/observation}
    \label{fig:spherical_directions}
   \end{center}
\end{minipage}

\bigskip \bigskip

\begin{center}
  \begin{tabular}{c@{\quad}c@{\quad}c}
    \includegraphics[width=0.31\linewidth]{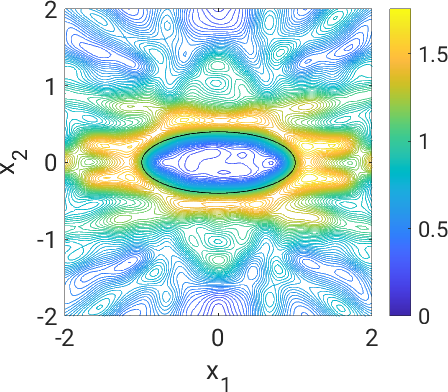}
    & \includegraphics[width=0.31\linewidth]{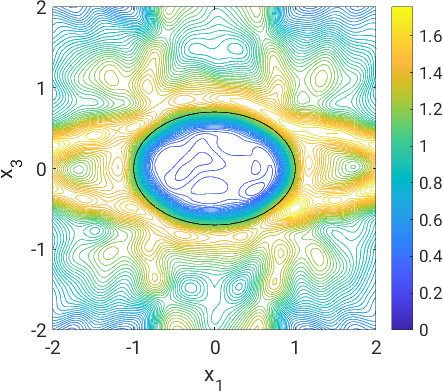}
    & \includegraphics[width=0.31\linewidth]{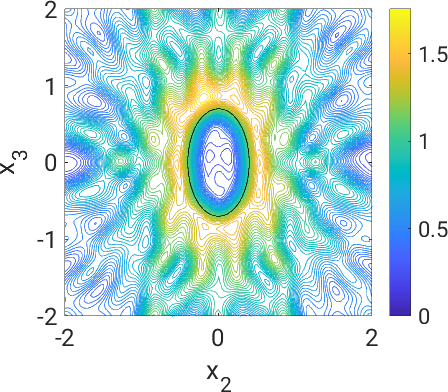} \\[1ex]
    (a) & (b) & (c)
  \end{tabular}
  \vspace{-1ex}

 \captionof{figure}{Plot of $I^{1/2}(\bz)$, ellipsoid obstacle, 80 incident directions in uniformly distributed over the unit sphere; (a) $x_3 = 0$; (b) $x_2 = 0$; (c) $x_1 = 0$.}
 \label{fig:ellipsoid_uniform}
\end{center}

\bigskip \bigskip

Finally, we also computed backscattering data for both obstacles for a set of $80$ incidence and observation directions that are distributed uniformly over the unit sphere. These directions are indicated in Figure \ref{fig:spherical_directions}. Compared to $40$ directions on some circle in one plane, these $80$ directions distributed on the unit sphere are actually much more sparse. The reconstructions of the ellipsoid using the indicator function $I^{1/2}$ displays in Figure \ref{fig:ellipsoid_uniform} are still acceptable even though a lot of points with large values of the indicator function are located relatively far from the obstacle. We are convinced that significantly better results are possible using more directions of incidence and, in particular, by using data for higher frequencies.

This holds even more so for the two sphere obstacle. Here, as shown in Figure \ref{fig:twospheres_xz_uniform} (a), high values of the indicator function no longer produces a recognizable reconstruction of the obstacle. Surprisingly, from Figures \ref{fig:ellipsoid_uniform} and \ref{fig:twospheres_xz_uniform}, an interesting observation that we cannot yet explain from our theory, is that the obstacle can be well captured if we consider the low values of the indicator function. For a single direction, the strip $S_D(\theta)$ containing the obstacle is clearly identified, as shown in Figure \ref{fig:twospheres_xz_uniform} (b). However, the complicated shape of the non convex obstacle leads to significant areas with high indicator values away from the boundary of the obstacle. This demonstrates the limits of the proposed method when data is only available for a modest frequency range.

\begin{figure}[t]
  \centering
  \begin{tabular}{c@{\qquad\qquad\qquad}c}
    \includegraphics[width=0.32\linewidth]{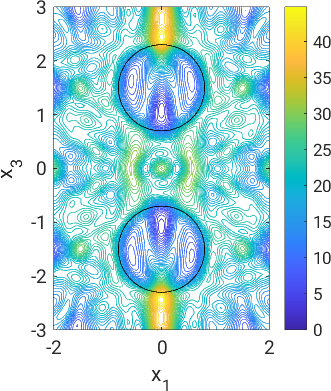}
    & \includegraphics[width=0.335\linewidth]{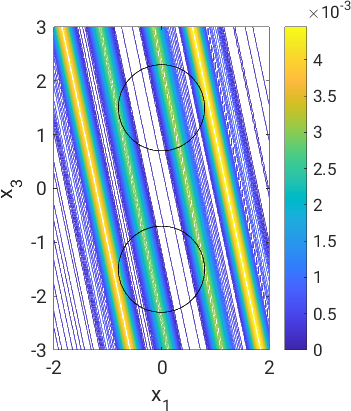} \\
    (a) & (b)
  \end{tabular}
  \vspace{-2ex}

 \caption{Reconstruction for two spheres with 80 uniformly distributed directions; (a) $I^{1/2}$ for all directions; (b) $I_{\bd}$ for $\bd \approx (0.960, -0.191, 0.203)^\top$.}
 \label{fig:twospheres_xz_uniform}
\end{figure}

\section{Conclusions}
In this paper we make a first step for inverse electromagnetic scattering problems with multi-frequency backscattering electric far field patterns taken at sparse directions. In particular, we show that the smallest strip with normal $\bd$ containing the unknown object can be approximately reconstructed by the multiple high frequency backscattering far field patterns taken at the observation directions $\pm\bd$. This establishes the theoretical basis for determining the location and shape of the object by backscattering far field data at sparse directions. One fast and robust direct sampling method is also proposed to recover the unknown object. Some numerical examples are designed to verify the effectiveness and robustness of the proposed method. The reconstructions further show that even the concave part can be well captured with an adequate number of observation directions.

For simplicity, we have focused on the perfect conductor in this paper. To deal with the nonlinearity of the inverse problem, we have considered the physical optics approximation with high frequencies for the theoretical analysis. However, the numerical reconstructions show that the proposed direct sampling method works very well even for moderately sized frequencies. To fully explain this observation, new tools and techniques have to be considered. We hope to deliver this in a forthcoming paper.

\section*{Acknowledgement}
The research of X. Ji is supported by the NNSF of China under grants 91630313 and 11971468,
and National Centre for Mathematics and Interdisciplinary Sciences, CAS.
The research of X. Liu is supported by the NNSF of China under grant 11971701, and the Youth Innovation Promotion Association, CAS.

\bigskip

%
%

\end{document}